\documentclass[a4paper,10pt]{amsart}
\usepackage{amsmath}
\usepackage{amsfonts}
\usepackage{amssymb}
\usepackage[all]{xy}

\usepackage[hypertex]{hyperref}

\newtheorem{theorem}{Theorem}[section]
\newtheorem{lemma}[theorem]{Lemma}

\newtheorem{corollary}[theorem]{Corollary}

\newtheorem{remark}[theorem]{Remark}

\newcommand{\modu}{\mathrm{mod}}
\newcommand{\ellk}{k}
\newcommand{\GA}{A}

\newcommand{\Z}{\mathbb{Z}}
\newcommand{\Q}{\mathbb{Q}}

\newcommand{\Hom}{\mathrm{Hom}}

\newcommand{\Rep}{{\textsc{Rep}}}

\newcommand{\C}{\mathbb C}

\newcommand{\cc}{\mathcal C}
\newcommand{\zz}{\mathcal Z}

\begin{document}
\title{The Tambara-Yamagami categories and 3-manifold invariants}
\author[V. Turaev]{Vladimir Turaev}
 \address{%
 Vladimir Turaev\newline
  \indent            Department of Mathematics \newline
\indent  Indiana University \newline
                     \indent Bloomington IN47405 \newline
                     \indent USA \newline
\indent e-mail: vtouraev@indiana.edu}
\author[L. Vainerman]{Leonid Vainerman}
\address{Leonid Vainerman\newline
     \indent         Department of Mathematics \newline
\indent    University of Caen\newline
                     \indent 14000 Caen \newline
                     \indent France \newline
\indent e-mail:  leonid.vainerman@math.unicaen.fr}
\date{\today}

\begin{abstract}
We prove that if two Tambara-Yamagami categories
$\mathcal{TY}(A,\chi,\nu)$ and $\mathcal{TY}(A',\chi',\nu')$  give
rise to the same state sum invariants of 3-manifolds and the order
of one of the groups $A, A'$ is odd, then $\nu=\nu'$ and  there is a
group isomorphism $A\approx A'$ carrying $\chi$ to $\chi'$.     The
proof is based on an explicit computation of the state sum
invariants for the lens spaces of type $(k,1)$.
\end{abstract}
\maketitle

\section*{Introduction}
\date{July 26, 2010}

One of the fundamental achievements of quantum topology was a discovery of
a  non-trivial connection  between monoidal categories and 3-manifolds. This connection was first observed by O. Viro and V. Turaev   and later generalized  in the papers of    J. Barrett, B. Westbury, A. Ocneanu, S. Gelfand,   D. Kazhdan and others. Their results  may be summarized by
saying that every spherical fusion category $\cc$ over $\mathbb C$ with $\dim (\cc)\neq 0$ gives rise to   a    topological invariant $\vert
M\vert_{\cc}\in \mathbb C$ of any closed oriented 3-dimensional  manifold $M$. A prototypical example of a spherical fusion category is   the category $\Rep (G)$ of finite-dimensional complex representations of a finite group $G$.  This category  allows nice operations on objects and morphisms: direct sums, tensor products, left and right dualization. Moreover,  $\Rep (G)$ 
 contains a finite family of \lq\lq simple" objects (=  irreducible representations) such that all objects split as  direct sums of the objects of this family.  Certainly, the sets of morphisms in $\Rep (G)$ are  finite-dimensional complex vector spaces.  Axiomatizing these properties, one obtains    a notion of a  fusion category, see  \cite{ENO}. The condition of sphericity \cite{BW2}  on a  fusion category $\cc$    says that  the \lq\lq left" and   \lq\lq right"   dimensions of   the objects of $\cc$ are equal.   The resulting $\mathbb C$-valued dimension of objects of $\cc$ is   preserved under isomorphisms in~$\cc$.  This allows one to define the dimension $\dim (\cc)$ of a  spherical fusion category $\cc$  as the
sum of the squares of the dimensions of the  isomorphism classes of 
simple objects. For $\cc=\Rep (G)$, the dimension of objects is the usual dimension of the underlying vector spaces and  $\dim (\Rep (G))=\vert G\vert$.  The class of spherical fusion categories includes the categories of type $\Rep (G)$   and many other categories some of which will be discussed below. 
 The class of spherical fusion categories is believed to be \lq\lq big  but not too big" so that one may hope for some kind of classification.

  The  invariant of a 3-manifold $M$ associated with $ \Rep (G)$ is nothing but the number of homomorphisms from the fundamental group of $M$ to $G$.   In general, the  invariant $\vert
M\vert_{\cc}$ associated  with a  spherical fusion category $\cc$ is not determined by  the fundamental group.  The definition of~$\vert M\vert_{\cc}$ is rather subtle and proceeds in terms of   state sums on a triangulation of~$M$.
The key algebraic ingredients of these state sums are the so-called
$6j$-symbols associated with~$\cc$.  

The formula 
$(M,\cc)\mapsto \vert M\vert_{\cc}$ defines  a pairing  between homeomorphism classes of closed oriented 3-manifolds and
spherical fusion categories of non-zero dimension. A study of this pairing leads to
natural questions both in algebra and topology. One usually studies  the topological aspects. Is the pairing $(M,\cc)\mapsto \vert M\vert_{\cc}$ sufficiently strong to distinguish the 3-sphere from   other 3-manifolds? (The answer is \lq\lq yes"). 
Is it sufficiently strong to distinguish arbitrary 3-manifolds up to homeomorphism? 
(The answer is \lq\lq no", see \cite{Fun}). 

We shall  focus on   algebraic questions and specifically on the following reconstruction problem:    To what extent a spherical fusion
category can be reconstructed from  the associated 3-manifold invariants? The rational for this problem is that the  number  $\vert
M\vert_{\cc}$ may be viewed as a generalized dimension  of $\cc$ determined by $M$.  The reconstruction problem is intriguing already for the categories of type $\Rep (G)$. Is it true that for any non-isomorphic finite groups $G_1, G_2$, there is a closed oriented 3-manifold $M$ such that the numbers of homomorphisms from $\pi_1(M)$ to $G_1$ and $G_2$ are different? We do not know the answer. 

In this paper, we
study   the reconstruction problem  for  an interesting class of spherical fusion
categories  recently introduced by Tambara and Yamagami \cite{TY}.
A   Tambara-Yamagami   category
$\mathcal{TY}(A,\chi,\nu)$ is determined by   a bicharacter $\chi$
on a finite abelian group $A$   and   a sign $\nu=\pm 1$. By a {\it
bicharacter} on $A$ we mean a  non-degenerate symmetric bilinear
pairing $\chi:A\times A\longrightarrow S^1$; the non-degeneracy of
$\chi$ means that the adjoint homomorphism $A\to \Hom (A, S^1)$ is
bijective. The pair $(A, \chi)$ will be called a {\it bicharacter
pair}.

Two bicharacter pairs $(A, \chi)$ and $(A', \chi')$   are said to be
isomorphic if there is an isomorphism $A \cong  A' $ transforming
$\chi $ into $\chi' $. It is known that two Tambara-Yamagami
categories, $\mathcal{TY}(A,\chi,\nu)$ and
$\mathcal{TY}(A',\chi',\nu')$, are monoidally
 equivalent if and only if  the pairs $(A,\chi)$ and $(A',\chi')$ are isomorphic and $\nu=\nu'$.

Each bicharacter pair  $(A, \chi)$ splits uniquely as an orthogonal sum
$$(A,\chi)= \bigoplus_{p} (A^{(p)}, \chi^{(p)}),$$
where $p$ runs over all prime natural numbers,
$A^{(p)}\subset A$ is the abelian $p$-group consisting of the elements of $A$ annihilated by a sufficiently big power of
$p$, and $\chi^{(p)}:A^{(p)}\times A^{(p)} \longrightarrow S^1$ is the restriction of $\chi$ to $A^{(p)}$.
 In the sequel, the order of a group $A$ is denoted $\vert A\vert$.

\begin{theorem}\label{thm-one}
Let $\cc=\mathcal{TY}(A,\chi,\nu)$ and $\cc'=\mathcal{TY}(A',\chi',\nu')$ be two Tambara-Yamagami categories such that $\vert M \vert_{\cc}= \vert M \vert_{\cc'}$ for all closed oriented 3-manifolds~$M$.

(a) We have $\vert A\vert=\vert A'\vert$ and if $\vert A\vert$ is not a positive power of $4$, then $\nu=\nu'$.

(b) For every odd prime $p$,  the pairs $(A^{(p)}, \chi^{(p)})$ and $(A'^{(p)}, \chi'^{(p)})$ are isomorphic.
\end{theorem}

Combining the claims (a) and (b) we obtain the following corollary.

\begin{corollary}\label{cor-one}
Let $\cc=\mathcal{TY}(A,\chi,\nu)$ and $\cc'=\mathcal{TY}(A',\chi',\nu')$ be two Tambara-Yamagami categories such that
 $\vert M \vert_{\cc}= \vert M \vert_{\cc'}$ for all closed oriented 3-manifolds~$M$. If $\vert A\vert$
   is odd, then the bicharacter pairs $(A , \chi )$ and $(A' , \chi' )$ are isomorphic and $\nu=\nu'$.
\end{corollary}

We conjecture a similar claim in the case where $\vert A\vert$ is even.

The proof of Theorem~\ref{thm-one}   is based on an explicit computation of $\vert M \vert_{\cc}$ for the lens spaces $L_k=L_{k,1}$ with $k=0,1, 2,\ldots$ Recall that $L_k $ is the closed oriented 3-manifold obtained from the 3-sphere $S^3$
 by surgery along a trivial knot in $S^3$  with   framing $k$. In particular, $L_0=S^1 \times S^2$, $L_1=S^3$, and $L_2=\mathbb
{R} { P}^3$. The manifolds  $\{L_k\}_{k}$ are pairwise non-homeomorphic; they are distinguished by the fundamental group
$\pi_1(L_k)=\Z/k\Z$.

To formulate our computation of $\vert L_k \vert_{\cc}$, we recall
the notion of a Gauss sum. Let $A$ be a finite abelian group and
$\chi:A\times A\longrightarrow S^1$ be  a symmetric bilinear form  (possibly
degenerate). A {\it quadratic map} associated with~$\chi$ is a map
$\mu:A\to S^1$ such that for all $a,b\in A$,
\begin{equation*}
\mu(a+b)=\chi(a,b)\, \mu(a)\, \mu(b) .
\end{equation*}
In other words,  the coboundary of $\mu$ is equal to $\chi$. Such a
$\mu$ always exists (see, for example, \cite{Klep}) and  determines
the normalized  {\it Gauss sum}
$$\gamma(\mu)= \vert A\vert^{-1/2} \vert A^\perp_\chi \vert^{-1/2} \sum_{a\in A }  \mu(a) \in \C,$$
where $$A^\perp_\chi=\{a\in A\, \vert \, \chi(a,b)=1 \,\ {\text {for
all}}\,\ b\in A\}$$ is the  annihilator of $\chi$. (If $\chi$ is a
bicharacter, then $A^\perp_\chi=\{0\}$.) The normalization is chosen
so that either $\gamma(\mu)=0$ or $\vert \gamma(\mu)\vert=1$ (see
Lemma~\ref{lemma-twominus} below).

Denote by $Q_\chi$ the set of quadratic maps associated with $\chi$.
This set has precisely $\vert A\vert$ elements; this follows from
the fact that any two quadratic maps associated with~$\chi$ differ
by a homomorphism $A\to S^1$. Every integer $k\geq 0$ determines a
subgroup $A_k=\{a\in A\, |\, ka =0\}$ of $A$ and a number
$$\zeta_k (  \chi)=  \vert A\vert^{-1/2} \vert A_k \vert^{-1/2} \sum_{ \mu\in Q_\chi}  \gamma(\mu)^k\in \C.$$
For example, $A_0=A$ and $\zeta_0(\chi)=1$.

\begin{theorem}\label{thm-two}
Let $\cc=\mathcal{TY}(A,\chi,\nu)$  be a Tambara--Yamagami category.
For any odd integer $k\geq 1$,   we have
\begin{equation}\label{eq-odd}
\vert L_{k} \vert_{\cc}= \frac {\vert A_{k}\vert }{2\vert A\vert} .
\end{equation}
For any even integer $k\geq 0$, we have
\begin{equation}\label{eq-even}
\vert L_{k}\vert_{\cc}=\frac{\vert A_{k}\vert +\nu^{k/2}\vert A\vert^{1/2}\vert A_{k/2}\vert^{1/2}\zeta_{k/2}
(\chi)}{2\vert A\vert}.
\end{equation}
\end{theorem}

For $k=0$, Formula~\eqref{eq-even} gives $\vert S^1 \times S^2 \vert_{\cc} =1$ which is known to be true for all spherical fusion categories $\cc$.

Our proof of Theorem~\ref{thm-two} is based on two results. The first is the equality $\vert M \vert_{\cc}= \tau_{\zz(\cc)} (M)$ recently established in \cite{TVi}. Here $\cc$ is an arbitrary spherical fusion category
of non-zero dimension, $\zz(\cc)$ is the Drinfeld-Joyal-Street center of $\cc$, and $\tau_{\zz(\cc)}(M)$ is the Reshetikhin-Turaev invariant of $M$. The second result is the computation of the center of $\cc=\mathcal{TY}(A, \chi,\nu)$ in \cite{GNN}.

The paper is organized as follows. In Section~\ref{sec-1} we recall the Tambara-Yamagami category and its center  and  prove Theorem~\ref{thm-two}. In Sections~\ref{sec-3} and~\ref{sec-3(a)} we prove respectively claims (a) and (b) of Theorem~\ref{thm-one}.

This paper was started during the visit of VT to the University of
Caen in June 2010. VT would like to thank the University of Caen for
hospitality. The work of V.~Turaev  was partially supported by the
NSF grant DMS-0904262.

\section{The   Tambara-Yamagami categories and their centers}\label{sec-1}

In this section,   $(A,\chi)$ is a bicharacter pair,  $\nu=\pm 1$,
and  $n=\vert A\vert$.

\subsection{The   category $\mathcal{TY}(A,\chi,\nu)$}\label{subsec-1}
The  simple objects of the Tambara-Yamagami category $\mathcal
C=\mathcal{TY}(A,\chi,\nu)$ are all elements $a$ of $ A$ and an
additional object $m$. The unit object of $\mathcal C$ is the zero
element  $0\in A$. All other objects of $\mathcal C$ are finite
direct sums of the simple objects. The tensor product in $\mathcal
C$ is determined by the following fusion rules:
$$
a\otimes b=a+b \, \, {\text {and}} \, \,  a\otimes m=m\otimes a=m \, \ {\text{for all}} \,\ a,b\in A,\ {\text {and}}\ \,
m\otimes m=\bigoplus_{a\in A} a.
$$
The category $\mathcal C$ is associative but generally speaking not strictly associative. For any simple objects $U,V,W$ of    $\mathcal C$, the associativity isomorphism $\phi_{ U,V,W}: (U\otimes V) \otimes W\to U \otimes (V \otimes W)$  is  given by the following formulas (where $a,b,c$ run over $A$):\\
$$\begin{array}{lcl}
\phi_{a,b,c} = id_{a+b+c},\hfill \phi_{a,b,m} = id_m,\hfill
\phi_{m,a,b} = id_m,\\ \phi_{a,m,b} = \chi(a,b)id_m,\hfill
\phi_{a,m,m} = \displaystyle\bigoplus_{b\in A}id_b,\hfill
\phi_{m,m,a} = \displaystyle\bigoplus_{b\in A}id_b,\\
\phi_{m,a,m} = \displaystyle\bigoplus_{b\in A}\chi(a,b)id_b,\,
\hfill \phi_{m,m,m} = ( {\nu} n^{-1/2}
\chi(a,b)^{-1}id_m)_{a,b}.
\end{array}$$
The unit isomorphisms are trivial. The duality in $\cc$ is defined
by $a^*=-a$ for  all $a\in A$ and  $m^*=m$. The left duality
morphisms in $\mathcal C$  are the identity maps $ 0\to a\otimes
a^*,\ a^*\otimes a\to 0$ for $a\in A$, the inclusion $
0\hookrightarrow  m\otimes m $  and $\nu  n^{1/2}$ times the obvious
projection $  m\otimes m \to 0$. The right duality morphisms in
$\mathcal C$ are the identity maps $ 0\to a^*\otimes a,\ a\otimes
a^*\to 0$ for $a\in A$, $\nu$ times the inclusion $ 0\hookrightarrow
m\otimes m $  and $n^{1/2}$ times the obvious projection $  m\otimes
m\to 0$.

We define a fusion category as a $\C$-linear  monoidal category with
compatible left and right dualities such that all objects are direct
sums of simple objects, the number of isomorphism classes of simple
objects is finite, and the unit object is simple. (An object $V$ is
simple if ${\rm End } (V)=\C \, id_V$.) The condition of sphericity
says that the left and right dimensions of all objects are equal.    A basic reference on the theory of fusion categories
is \cite{ENO}. It is easy to see that $\cc=\mathcal{TY}(A,\chi,\nu)$
is a spherical fusion category of dimension $2n$.

\subsection{The   center}\label{subsec-2}  The center $ {\zz(\cc)} $ of $\mathcal C=\mathcal{TY}(A,\chi,\nu)$  was computed in \cite{GNN},
Prop.\@ 4.1. The category $ {\zz(\cc)} $ has  three types of   simple objects whose description together with
the corresponding quantum dimensions and twists is as follows:

(1) $2n$ invertible objects $X_{(a,\varepsilon)}$, where $ a$ runs
over $ A$ and $ \varepsilon$ runs over complex square roots of
$\chi(a,a)^{-1}$. Here $\dim(X_{(a,\varepsilon)})=1$ and $
\theta_{(a,\varepsilon)}=\chi(a,a)^{-1}$;

(2) $\frac{n(n-1)}{2}$ objects $Y_{(a,b)}$ parameterized by
unordered pairs $(a,b)$, where $a,b\in A,\ a\neq b$. Here
$\dim(Y_{(a,b)})=2$ and $ \theta_{(a,b)}=\chi(a,b)^{-1}$;

(3) $2n$ objects $Z_{(\mu,{\Delta})}$, where $\mu $ runs over
$Q_\chi$ and ${\Delta}$ runs over the square roots of $ {\nu}
\gamma(\mu)$. Here $\dim(Z_{(\mu,{\Delta})})=  n^{1/2}$ and
$\theta_{(\mu,{\Delta})}={\Delta}$.

Denote by $I$ the set of the (isomorphism classes of) simple objects of $\zz(\cc)$. The dimension
of $\zz(\cc)$ is computed by
$$\dim \zz(\cc)=\sum_{i\in I} (\dim(i))^2=2n \times 1 + \frac{n(n-1)}{2} \times 4+ 2n \times n=4n^2.$$
We will need the following more general computation.

\begin{lemma}\label{lemma-one}
For an  integer $k\geq 0$, set
$
\tau_{k}= \sum_{i\in I}\theta_i^{k}(\dim(i))^2
$,
where $\theta_i$ and $\dim(i)$ are the twist and the dimension of $i\in I$.
If $k$ is odd, then  $ \tau_{k} = 2n \vert A_{k}\vert$.   If $k$ is even, then $
\tau_{k}  = 2n ({\vert A_{k}\vert} + {\nu^{k/2}}\vert A\vert^{1/2}\vert A_{k/2}\vert^{1/2}
\zeta_{k/2} (\chi)) $.
\end{lemma}

\begin{proof} A direct computation shows that $
\tau_{k}=2u_k+nv_k$, where    $$
u_k=\sum_{a\in A}\chi(a,a)^{-k}+\sum_{(a,b)\in A^2, a\neq b}\chi(a,b)^{-k} $$ and   $v_k=\sum_{({\mu}, \Delta) }{\Delta}^k $.
Since $\chi$ is non-degenerate,
$$
u_k=\sum_{a,b\in A}  \chi(a,b)^{-k}=\sum_{a,b\in A}\chi(a,b^{-k}) =n \, |A_k|.
$$
If $k$ is odd, then the contributions of the pairs $({\mu}, \Delta)$ and $({\mu}, -\Delta)$ to $v_k$ cancel each other so that $v_k=0$ and  $ \tau_{k} = 2n \vert A_{k}\vert$.
For  even  $k$,
$$v_k=\sum_\mu 2  ({\nu} \gamma (\mu) )^{k/2}=2 {\nu^{k/2}}\vert A\vert^{1/2}\vert A_{k/2}\vert^{1/2}\zeta_{{k/2}} (\chi).$$
\end{proof}

\subsection{Proof of Theorem \ref{thm-two}}\label{sec-2}
Since $\cc=\mathcal{TY}(A,\chi,\nu)$ is a spherical fusion category
of non-zero dimension, it determines for any closed oriented
3-manifold $M$  a   state sum invariant  $\vert M\vert_{\cc}\in \C$,
see \cite{TV}, \cite{BW}. By a theorem of M\"uger \cite{Mug}, the
category $\zz(\cc)$ is modular in the sense of~\cite{Tur}. A modular
category endowed with a square root $\mathcal D$ of its dimension
gives rise to the Reshetikhin-Turaev invariant of any $M$ as above.
The RT-invariant of $M$ determined by $\zz(\cc)$ and the square root
$\mathcal D=2n=\tau_1$ of $\dim \zz(\cc)$ will be denoted by
$\tau_{\zz(\cc)}(M)$. A theorem of Virelizier and Turaev \cite{TVi}
implies that  $\vert M\vert_{\cc}=\tau_{\zz(\cc)}(M)$ for all $M$.
By \cite{Tur}, Chapter II, 2.2, for all $k\geq 0$,
\begin{equation*}
\tau_{\zz(\cc)}(L_{k})={\mathcal D}^{-2}\sum_{i\in I}\theta_i^{k}(\dim(i))^2=4n^{-2} \tau_k.
\end{equation*}
Substituting the expression for $\tau_k$ provided by Lemma~\ref{lemma-one}, we obtain the claim of the theorem.

\section{Proof of Theorem~\ref{thm-one}(a)}\label{sec-3}

We start with a well known lemma. In this lemma we call a quadratic
map $\mu :A\to S^1$ {\it homogeneous} if $\mu(na)=(\mu(a))^{n^2}$
for all $n\in \Z$ and $a\in A$.

\begin{lemma}\label{lemma-twominus}   Let $A$ be a finite abelian group and $\mu :A\to S^1 $ be a quadratic map
  associated with a symmetric bilinear form $\chi:A\times A \to
  S^1$. Set $A^\perp=A^\perp_\chi\subset A$.

  - If $\mu(A^\perp)\neq 1$, then $\gamma(\mu)= 0$.

  - If $\mu(A^\perp)= 1$, then $\vert
  \gamma(\mu)\vert=1$.

  - If $\mu(A^\perp)= 1$ and $\mu$ is homogeneous, then
  $\gamma(\mu)$ is an 8-th complex root of unity.
\end{lemma}

\begin{proof}
We have
$$ \vert  A\vert \, \vert A^\perp \vert\,  \vert \gamma(\mu)\vert^2= \vert \sum_{a\in A} \mu(a)\vert^2=
\sum_{a, b\in A}\mu(a)\overline {\mu(b)} =\sum_{a, b\in A}\mu(a){\mu(b)}^{-1}$$
$$=\sum_{a, b\in A} \mu(a+b) {\mu(b)}^{-1} =\sum_{a, b\in A} \chi (a,b)  \mu(a).$$
When $b$ runs over $A$, the complex number $\chi(a,b)$ runs over a finite subgroup of $S^1$.
We have  $\sum_{b\in A} \chi(a,b)=0$ unless this subgroup is trivial. The latter holds if and only if
$a\in A^\perp$ and in this case $\sum_{b\in A}  \chi(a,b)=\vert A\vert$. Therefore,
$$\vert  A\vert \, \vert A^\perp \vert\,  \vert \gamma(\mu)\vert^2=
\vert  A\vert \, \sum_{a\in A^\perp} \mu(a).$$ The restriction of
$\mu$ to $A^\perp$ is a group homomorphism $A^\perp\to S^1 $. If
$\mu(A^\perp)\neq 1$, then $\sum_{a\in A^\perp} \mu(a)=0$ and
therefore $\gamma(\mu)=0$. Suppose now that   $\mu(A^\perp)= 1$.
Then $\sum_{a\in A^\perp} \mu (a)=\vert A^\perp\vert$ and therefore
$\vert
  \gamma(\mu)\vert=1$. The equality  $\mu(A^\perp)= 1$ also ensures that $\mu$ is the composition of the projection $A\to A'= A/A^\perp$
   with a quadratic map $\mu':A'\to S^1 $
  associated with the non-degenerate symmetric bilinear form $ A'\times A' \to S^1$ induced by $\chi$.
   It follows from the definitions that $\gamma(\mu)=\gamma(\mu')$. If $\mu$ is homogeneous, then so is $\mu'$. It is known (see, for instance, \cite{Sc}, Chapter 5, Section 2)
   that  for any  homogeneous quadratic map on a finite abelian group associated with a non-degenerate symmetric bilinear form, the corresponding  invariant $\gamma$ is an $8$-th root of unity.
    This implies the last claim  of the lemma.
\end{proof}

\begin{lemma}\label{lemma-two} Let $(A,\chi)$ be a bicharacter pair.
  For any integer $\ellk \geq 1 $, either
$\zeta_\ellk (\chi)=0$ or $  \zeta_\ellk (\chi) $ is an $8$-th root
of unity. If $\ellk=1$ or $\ellk$ is divisible by $8\vert A\vert$,
then $\zeta_\ellk(\chi)=1$.
\end{lemma}
\begin{proof} Pick a
  quadratic  map $\mu_0:A\to S^1$ associated with $\chi$.
Observe that for every integer $k $, the function $\mu_0^{k}:A\to
S^1$ carrying any $ c\in A$ to $ (\mu_0(c))^{\ellk}$ is a  quadratic
map associated with the symmetric  bilinear form $\chi^{k}:A\times
A\to S^1$ defined by $\chi^{k}(a,b)= (\chi (a,b) )^{k}$. We claim
that for all $k\in \Z$,
\begin{equation}\label{prin}\zeta_\ellk(\chi)=\gamma (\mu_0^{-k})\,
(\gamma (\mu_0))^k.
\end{equation} Indeed, since $\chi$ is non-degenerate, any quadratic map
$\mu:A\to S^1$ associated with $\chi$   can be expanded in the form
$\mu(a)= \chi (a, c)\, \mu_0(a)$ for a unique $c=c(\mu)\in A$. Since
$\chi(a,c) \, \mu_0(a)=\mu_0(a+c)\, \mu_0(c)^{-1}$ for all $a,c\in
A$, we have
$$
 \zeta_\ellk(\chi)= |A|^{-1/2} |A_\ellk|^{-1/2}\sum_{\mu\in Q_\chi}(|A|^{-1/2}\sum_{a\in A}
\mu(a))^\ellk
$$
$$
= |A|^{-1/2} |A_\ellk|^{-1/2} \sum_{c\in A} (|A|^{-1/2}\sum_{a\in
A}\chi(a,c)\, \mu_0(a))^\ellk$$ $$= \{ |A|^{-1/2} |A_\ellk|^{-1/2}
\sum_{c\in A}\mu_0(c)^{-\ellk}\}\{|A|^{-1/2}\sum_{b\in
A}\mu_0(b)\}^{\ellk}$$
$$=\gamma (\mu_0^{-k})\, (\gamma (\mu_0))^k.
$$
In the last equality we use  the obvious fact that
$A^\perp_{\chi^{-k}}=A_k$.

We can always choose  $\mu_0:A\to S^1$ to be homogeneous. Then
$\mu_0^{-k}$ also is homogeneous. Since $\chi$ is non-degenerate,
the previous lemma implies that $\gamma (\mu_0)$ is an $8$-th root
of unity and $\gamma (\mu_0^{-k})$ is either zero or an $8$-th root
of unity. This implies the first claim of the lemma.

For $k=1$, Formula~\eqref{prin} gives
$$\zeta_1(\chi)= \gamma (\mu_0^{-1})\, \gamma (\mu_0) ={\gamma
(\overline{\mu_0})} \, \gamma (\mu_0)=\overline{\gamma (\mu_0)} \,
\gamma (\mu_0)= 1,$$ where the
 overbar  is the complex conjugation.

Observe that $\mu_0^{2n}=1$ for $n=\vert A\vert$. Indeed,  for any
$a\in A$,
 $$1=\mu_0(0)=\mu_0(2na)=  (\mu_0(a))^{2n} \chi(a,a)^{n(n-1)}$$
 $$=
 (\mu_0(a))^{2n}
 \chi(na,(n-1) a)=    (\mu_0(a))^{2n}.$$
 Therefore for all   $k\in 2n\Z$, we have $\gamma
 (\mu_0^{-k})=1$. If $k \in 8 \Z$, then   $(\gamma
 (\mu_0))^k=1$. Hence, if $k\in 8n\Z$, then $\zeta_\ellk(\chi)=\gamma (\mu_0^{-k})\,
(\gamma (\mu_0))^k=1$.
\end{proof}

\subsection{Proof of Theorem~\ref{thm-one}(a)}
For $k=1$, Formula~\eqref{eq-odd} gives $\vert L_{1}\vert_{\cc}=(2\vert A\vert )^{-1}$. Thus,
$$\vert A\vert =\vert L_{1}\vert_{\cc}^{-1}/2=\vert L_{1}\vert_{\cc'}^{-1}/2=\vert A'\vert.$$
This and Formula~\eqref{eq-odd} implies that $\vert A_k\vert =\vert
A'_k\vert$ for all odd $k\geq 1$.

Set $n=\vert A\vert=\vert A'\vert $. Suppose that $\nu\neq\nu'$. Assume for concreteness that $\nu=-1$
and $\nu'=+1$. Formula~\eqref{eq-even} with $k=2$ and Lemma~\ref{lemma-two} show that
$$\vert A_{2}\vert -n^{1/2}=2n\vert L_{2}\vert_{\cc}=2n\vert L_{2}\vert_{\cc'}=\vert A'_{2}\vert + n^{1/2}.$$
Thus, $\vert A_{2}\vert -\vert A'_{2}\vert= 2n^{1/2}$. Therefore,
$n=m^2$ for an integer $m\geq 1$. Since $n$ is not a positive power
of $4$,   either $m=1$ or $m$ is  not a power of $2$. If $m=1$, then
$A=A'=\{0\}$ and so $A_{2}=  A'_{2}=\{0\}$ which contradicts the
equality $\vert A_{2}\vert -\vert A'_{2}\vert= 2m$.

Suppose that $m=n^{1/2}$ is  not a power of $2$. Pick an odd divisor
$\ell\geq 3$   of $m $. Applying Formula~\eqref{eq-even} to
$k=2\ell$, we obtain that
$$\vert A_{k}\vert - m\vert A_{\ell} \vert^{1/2}  \zeta_{\ell} (  \chi)=\vert A'_{k}\vert + m \vert A'_{\ell} \vert^{1/2}  \zeta_{\ell} (  \chi').$$
Note that $\vert A_{k}\vert=\vert A_{2}\vert\, \vert A_{\ell}\vert$
and similarly for $A'$. Since $\ell$ is odd, we have $\vert
A_\ell\vert =\vert A'_\ell \vert$. Therefore
$$\vert A_{2}\vert -\vert A'_{2}\vert= m \vert A_{\ell} \vert^{-1/2}  (\zeta_{\ell} (  \chi')+\zeta_{\ell} (  \chi)).$$
The right-hand side of this equality must be a real number that cannot exceed $2m \vert A_{\ell} \vert^{-1/2}$ by Lemma~\ref{lemma-two}. Thus, $\vert A_{2}\vert -\vert A'_{2}\vert \leq 2m \vert A_{\ell} \vert^{-1/2}$.
Since $\ell$ divides $n$, we have $A_\ell\neq 1$ so that $ \vert A_{\ell} \vert \geq  2$. This
 gives $\vert A_{2}\vert -\vert A'_{2}\vert \leq 2m /\sqrt{2}$ which contradicts the equality
  $\vert A_{2}\vert -\vert A'_{2}\vert= 2m$. This contradiction shows that $\nu=\nu'$.

\subsection{Remarks} (i) It is easy to extend the argument above to show that the conclusion of Theorem~\ref{thm-one}(a) holds also for $\vert A\vert=4$.

(ii) Let in the proof above $\vert A\vert=\vert A'\vert=n$ be  a
positive power of $2$ and $\nu=-1,\nu'=1$. Formula~\eqref{eq-even}
with $k=2\ell$, where $\ell\geq 3$ is odd, shows that
$$\vert A_{2\ell}\vert-n^{1/2}\vert A_{\ell}\vert^{1/2}\zeta_{\ell}(\chi)=2n\vert L_{2\ell}\vert_{\cc}=
2n\vert L_{2\ell}\vert_{\cc'}=\vert A'_{2\ell}\vert+n^{1/2}\vert
A'_{\ell}\vert^{1/2}\zeta_{\ell}(\chi').$$ But now $A_{\ell}=\{0\}$,
so $\vert A_{\ell}\vert=1$, $\vert A_{2\ell}\vert=\vert A_{2}\vert$
and similarly for $A'$. This  gives $\vert A_{2}\vert- \vert
A'_{2}\vert=n^{1/2}(\zeta_{\ell}(\chi')+\zeta_{\ell}(\chi))$.
Comparing with the equality $\vert A_{2}\vert-\vert A'_{2}\vert=
2n^{1/2}$ obtained above, we conclude that
$\zeta_{\ell}(\chi')+\zeta_{\ell}(\chi)=2$. By Lemma~\ref{lemma-two},
this is possible if and only if $\zeta_{\ell}(\chi)=
\zeta_{\ell}(\chi')= 1$  for all odd $l\geq 3$.

(iii)   The number $\zeta_k(\chi)$ is closely related to the
Frobenius-Schur indicator  $\nu_{2k}(m)$ of the   object $m$ of the
  category $\cc=\mathcal{TY}(A,\chi,\nu)$ computed by Shimizu \cite{Sh}.
Indeed, substituting $n=2k, V=m$ in    formula (3) of \cite{Sh} and
taking into account that $\dim(\mathcal C)=2|A|$, $
\theta_m=\Delta$, $ \dim(m)= |A|^{1/2}$, we obtain
$$
\nu_{2k}(m)=\frac{1}{2|A|^{1/2}}\sum_{\mu,\Delta}\Delta^{2k}=  |A|^{-1/2}\sum_{\mu\in Q_{\chi}} [\nu\gamma(\mu)]^k=
\nu^k |A_k|^{1/2} \zeta_k(\chi)
$$
(our sign $\nu$ is equal  to Shimizu's $sgn(\tau)$).  This and
Lemma~\ref{lemma-two} give another proof of the following results of
Shimizu (see \cite{Sh}, Theorem 3.5): the number $|A_k|^{-1/2}
\nu_{2k}(m) $ is either $0$ or an 8-th complex root of unity for all
$k$;   this number is   $0$ if and only if for some (and then for
any) $\mu\in Q_\chi$, there is $a=a_\chi\in A_k$ such that
$\mu(a)^k=1$. The latter claim follows   from
Lemma~\ref{lemma-twominus}, Formula~\eqref{prin}, and the equality
$A^\perp_{\chi^{-k}}=A_k$.

\section{Proof of Theorem~\ref{thm-one}(b)}\label{sec-3(a)}

\subsection{Preliminaries on bicharacters}\label{prelimbich} Any finite abelian group $\GA$ splits uniquely as a direct sum $\GA=\oplus_{p}\, \GA^{(p)}$, where $p\geq 2$
runs over all prime integers and $\GA^{(p)}$ consists of all
elements of $\GA$ annihilated by a sufficiently big power of $p$.
The group $\GA^{(p)}$ is a $p$-group, i.e., an abelian group
annihilated by a sufficiently big power of $p$.  Given a bicharacter
$\chi$ of $\GA$, we   have $\chi (\GA^{(p)}, \GA^{(p')})=1$ for any
distinct $p,p'$. Therefore the restriction, $\chi^{(p)}$, of $\chi$
to $\GA^{(p)}$ is a bicharacter and we have an orthogonal splitting
$(\GA,\chi)=\oplus_{p} \, (\GA^{(p)} , \chi^{(p)})$.

Fix a prime integer $p\geq 2$ and recall the properties of
bicharacters on $p$-groups, see, for example, \cite{De} for a
survey.   Given a bicharacter $\chi$ on a finite abelian $p$-group
$\GA$, there is an orthogonal splitting $(\GA,\chi)=\oplus_{{s}\geq
1} (\GA_{s}, \chi_{s})$, where $\GA_{s}$ is a direct sum of several
copies of $\Z/p^{{s}}\Z$ and $\chi_{s}:\GA_{s} \times \GA_{s}\to
S^1$ is a bicharacter. The rank of $\GA_{s}$ as a
$\Z/p^{{s}}\Z$-module depends only on $\GA$ and is denoted
${r}_{p,s} (\GA)$.

Assume from now on that $p\neq 2$. Then the splitting
$(\GA,\chi)=\oplus_{{s}\geq 1} (\GA_{s}, \chi_{s})$ is unique up to
isomorphism and each   $\chi_{s}$ is  an orthogonal sum of
bicharacters on ${r}_{s} (\GA)$ copies of the cyclic abelian group
$\Z/p^{{s}}\Z$.  Using the canonical injection
$\Z/p^{{s}}\Z\hookrightarrow S^1, z\mapsto e^{2\pi i z/p^{{s}}}$, we
can view $\chi_{s}$ as a pairing with values in the
ring~$\Z/p^{{s}}\Z$. This allows us to consider the determinant
$\det\chi_{s}\in \Z/p^{{s}}\Z$ of $\chi_{s}$. Since $\chi_{s}$ is
non-degenerate, $\det \chi_{s} $  is coprime with $p$. Let
$$\sigma_{p, s}(\chi)=(\frac{\det \chi_{s} }{p})  \in \{\pm 1\}
$$ be the corresponding Legendre symbol.  Recall that for    an
integer $d$ coprime with~$p$,   the   Legendre symbol
$(\frac{d}{p})$ is equal to $1$ if $d \,(\modu\ p)$ is a quadratic
residue
  and to $-1$
  otherwise, see, for example, \cite{IR}.
   If ${r}_{p,s} (\GA )=0$,
then by definition $\sigma_{p, s}=1$. It follows   from the
definitions that the integers $\{{r}_{p,s}\}_{s}$ are additive and
the signs $\{\sigma_{p,s}\}_{s}$ are multiplicative with respect to
orthogonal summation of bicharacter pairs. A theorem due to H.
Minkowski, E. Seifert, and C.T.C.\ Wall says that these invariants
form a complete system:  two bicharacters, $\chi_1$ and $\chi_2$, on
$p$-groups $\GA_1$ and $\GA_2$, respectively, are isomorphic if and
only if ${r}_{p,s}(\GA_1)={r}_{p,s}(\GA_2)$ and
$\sigma_{p,s}(\chi_1)=\sigma_{p,s}(\chi_2)$ for all ${s}\geq 1$.

For shortness, when  $p$ is specified, we   denote ${r}_{p,s} (\GA)$
and $\sigma_{p, s}(\chi)$ by $r_s(\GA)$ and $\sigma_{ s}(\chi)$,
respectively.

\subsection{Computation of $\zeta_k$}\label{compos}  Consider the $\C$-valued invariants $\{\zeta_k\}_{k\geq 1}$ of
bicharacters defined in the introduction. It is easy to deduce from
the definitions that $ \zeta_k(\chi\oplus\chi') = \zeta_k(\chi)\,
\zeta_k(\chi')  $  for any bicharacters $\chi, \chi'$  and any $k$.
Thus, the formula $\chi\mapsto\zeta_k(\chi)$ defines a
multiplicative function from the semigroup of bicharacter pairs
(with operation being the orthogonal sum $\oplus$) to $\C$.

Fix an odd prime   $p \geq 3$. We now compute $\zeta_k$ on the
bicharacters on $p$-groups. For any
 odd  integer $a$, set    $\varepsilon_{a}=i=\sqrt{-1}$ if $a\equiv 3 (\modu\ 4)$ and $\varepsilon_{a}=1$ otherwise.
For any integers $  \ellk, s \geq
  1$, we have $gcd(\ellk,p^s)=p^t$ with
$0\leq t \leq s$.   Set
$$ \alpha_{ \ellk,s}= \ellk s + s-t\quad {\text {and}} \quad
\beta_{ \ellk,s}=
\frac{\varepsilon^\ellk_{p^s}}{\varepsilon_{p^{s-t}}} \,
(\frac{h}{p})^{\ellk s + s-t} \, (\frac{\ellk'}{p})^{s-t}  \in \{\pm
1, \pm i\},$$ where $h=({p^s+1})/{2}\in \Z$ and $\ellk'=
{\ellk}/{p^t}\in \Z$.   Note that $gcd(h,p)=1$ so that the Legendre
symbol $(\frac{h}{p})$ is defined. If $t<s$, then $gcd(\ellk',p)=1$
so that the Legendre symbol $(\frac{\ellk'}{p})$ is defined; if
$t=s$, then by definition, $(\frac{\ellk'}{p})^{s-t}=1$.

\begin{lemma} \label{zetal}
For any $  \ellk \geq 1$  and any bicharacter $\chi$ on   a
$p$-group $\GA$,
\begin{equation} \label{zetal1}
\zeta_\ellk(\chi)= \prod_{s\geq 1} \, \beta_{ \ellk, s}^{r_{s}(\GA)}\, [\sigma_{s} (\chi)]^{\alpha_{ \ellk,s}}.
\end{equation}
\end{lemma}
\begin{proof}
The proof   is based on the following classical Gauss formula: for
any integer $d$ coprime with $p$,
\begin{equation} \label{gauss}
\sum_{j=0}^{p^s-1} \exp(\frac{2\pi i}{p^s}dj^2) =p^{\frac{s}{2}}\varepsilon_{p^{s}} (\frac{d}{p})^{s}.
\end{equation}
A more general formula holds for any integer $d$: if
$gcd(d,p^s)=p^t$ with $0\leq t\leq s$ and $d'= {d}/{p^t}$, then
\begin{equation} \label{gauss+}
\sum_{j=0}^{p^s-1} \exp(\frac{2\pi i}{p^s}dj^2)=p^t\sum_{j=0}^{p^{s-t}-1} \exp(\frac{2\pi i}
{p^{s-t}}d'j^2)=p^{\frac{s+t}{2}}\varepsilon_{p^{s-t}} (\frac{d'}{p})^{s-t},
\end{equation}
where, by definition, for $t=s$,   the expression
$(\frac{d'}{p})^{s-t}$ is equal to 1.

We  now prove   \eqref{zetal1}. It is clear that both sides of
\eqref{zetal1} are multiplicative with respect to orthogonal
summation of bicharacters. The results stated in
Section~\ref{prelimbich} allow us to reduce the proof of
\eqref{zetal1} to the case where    $\GA=\Z/p^s \Z$ for some $s\geq
1$. We must prove that for any bicharacter $\chi:\GA\times \GA\to
S^1$,
\begin{equation} \label{zetal1--}
\zeta_\ellk(\chi)=   \beta_{ \ellk, s} \, [\sigma_{s} (\chi)]^{\alpha_{ \ellk,s}}.
\end{equation}

Set as above $h=({p^s+1})/{2} $ and $\ellk'= {\ellk}/{p^t} $, where
$gcd(\ellk,p^s)=p^t$ with $0\leq t \leq s$. The bicharacter $\chi$
is given by $\chi(a,b)= \exp(\frac{2\pi i}{p^s}\Delta ab)$ for all $
a,b\in \GA$, where $\Delta$ is an integer coprime with $p$. Observe
that the map $\mu_0:A\to S^1$ carrying any $a\in A$ to $
\exp(\frac{2\pi i}{p^s}h\Delta a^2)$   is a quadratic map associated
with  $ \chi$. Formula (\ref{gauss})  and the multiplicativity of
the Legendre symbol imply that
$$
\gamma(\mu_0 )=p^{-s/2}\sum_{j=0}^{p^s-1} \exp(\frac{2\pi i}{p^s}h\Delta j^2)=
\varepsilon_{p^{s}} (\frac{h}{p})^{s} (\frac{\Delta}{p})^{s}
=
\varepsilon_{p^{s}} (\frac{h}{p})^{s}\, [\sigma_s(\chi)]^s.
$$
Similarly, Formula (\ref{gauss+}) implies that
$$
 \sum_{c\in A}\mu_0(c)^{-\ellk}= \sum_{j=0}^{p^s-1}\overline{\exp(\frac{2\pi i}{p^s}\ellk h\Delta j^2)}=
p^{(s+t)/2} {\varepsilon_{p^{s-t}}^{-1}}(\frac{h}{p})^{s-t}(\frac{\ellk'}{p})^{s-t} [\sigma_s(\chi)]^{s-t}.
$$
Since $|A|=p^s$ and $|A_\ellk|=gcd(\ellk,p^s)=p^t$, we have
$$\gamma(\mu_0^{-k})=
|A|^{-1/2} |A_\ellk|^{-1/2}\sum_{c\in A}\mu_0(c)^{-\ellk}=
{\varepsilon_{p^{s-t}}^{-1}}(\frac{h}{p})^{s-t}(\frac{\ellk'}{p})^{s-t}
[\sigma_s(\chi)]^{s-t}.$$ These computations and
Formula~\eqref{prin} imply that  $$ \zeta_\ellk(\chi)= \gamma
(\mu_0^{-k})\, (\gamma (\mu_0))^k=
\frac{\varepsilon^\ellk_{p^s}}{\varepsilon_{p^{s-t}}}
(\frac{h}{p})^{\ellk s + s-t}(\frac{\ellk'}{p})^{s-t}
[\sigma_s(\chi)]^{\ellk s + s-t}.
$$
This is equivalent to Formula~\eqref{zetal1--}.
\end{proof}

Note one   special case of Lemma~\ref{zetal}: if $k$ is divisible by
$2\vert A\vert$, then $\zeta_\ellk(\chi)= \prod_{s\geq 1} \, \beta_{
\ellk, s}^{r_{s}(\GA)}$. Indeed, in this case for all $s$ such that
$\Z/p^s \Z$ is a direct summand of $A$, we have $gcd(\ellk,p^s)=p^s$
and $\alpha_{ \ellk,s}  = \ellk s \in 2\Z$. For all other $s$, we
have $\sigma_{s} (\chi)=1$. Therefore $[\sigma_{s} (\chi)]^{\alpha_{
\ellk,s}}=1$ for all $s$.

\subsection{Proof of Theorem~\ref{thm-one}(b)} We begin with a few  remarks concerning the subgroups $(A_k)_{k }$
of $A$ defined in the introduction. Using the splitting $A=\oplus_p
\, A^{(p)}$, one easily checks that $A_{kl}=A_k \oplus A_l$ for any
relatively prime integers $k, l$. For any prime $p$, the integers
$(\vert A_{p^m}\vert)_{m\geq 1}$ depend only on the group $A^{(p)}$
and determine the isomorphism class of $A^{(p)}$.   Indeed,
$A^{(p)}=\oplus_{s\geq 1}(\Z/p^s\Z)^{r_{p,s}}$ for
$r_{p,s}=r_{p,s}(A)\geq 0$. Given $m\geq 1$,
$$A_{p^m}=(A^{(p)})_{p^m}= \oplus_{s= 1}^m (\Z/p^s\Z)^{r_{p,s}}\, \oplus \, \oplus_{s>m}(\Z/p^m\Z)^{r_{p,s}}.$$
Hence,
$$ \log_p (\vert A_{p^{m+1}}\vert/\vert A_{p^m}\vert)=r_{p,m+1}+r_{p,m+2} + \cdots .$$
Therefore, the sequence $(\vert A_{p^m}\vert)_{m\geq 1}$ determines
  the sequence $\{r_{p,s}(A)\}_{s\geq 1}$ and so determines the isomorphism
type of $A^{(p)}$.

Formula~\eqref{eq-odd} and the assumptions of the theorem imply that, for all odd $k\geq 1$,
$$\vert A_k\vert=2n\vert L_{k} \vert_{\cc}=2n \vert L_{k}\vert_{\cc'}=\vert A'_k\vert,$$
where $n=\vert A\vert=\vert A'\vert $. By the previous paragraph,
$A^{(p)}\cong A'^{(p)}$ for all   prime  $p \neq 2$, and for all
$s\geq 1$,
\begin{equation}\label{tok}
r_{p,s}(A^{(p)})=r_{p,s}(A)=r_{p,s}(A')=r_{p,s}(A'^{(p)}).
\end{equation}
  Since $n =\prod_{p\geq 2} \vert A^{(p)}\vert $, we also have
$\vert A^{(2)}\vert=\vert A'^{(2)}\vert$.

Let $N\geq 2$ be a  positive   power of 2 annihilating both
$A^{(2)}$ and $A'^{(2)}$. Then $A_{{N}}= A^{(2)}$ and $A'_{{N}}=
A'^{(2)}$. For any odd integer $\ell\geq 1$,
$$\vert A_{{N}\ell}\vert=\vert A_{{N}}\vert\, \vert A_\ell\vert=\vert A^{(2)}\vert\, \vert A_\ell\vert= \vert A'^{(2)}\vert\,
 \vert A'_\ell\vert=\vert A'_{{N}}\vert\,\vert A'_\ell\vert=\vert A'_{{N}\ell}\vert.$$
Similarly, $\vert A_{2N\ell}\vert=\vert A'_{2N \ell} \vert$.
Applying \eqref{eq-even} to $k=2N \ell$, we obtain   $\zeta_{ {N}
\ell} (\chi)=\zeta_{ {N} \ell} (\chi')$.

Fix from now on an odd prime $p$. The identity \eqref{tok} shows
that
 to prove that the bicharacter pairs
$(A^{(p)}, \chi^{(p)})$ and $(A'^{(p)}, \chi'^{(p)})$
  are isomorphic, it is enough to  verify that $\sigma_{ s}
  (\chi^{(p)})=\sigma_{ s}
  (\chi'^{(p)})$ for all $s\geq 1$.  Set
  $$\ell=   \frac{\vert A \vert}{\vert A^{(2)}\vert \vert A^{(p)}\vert}=
  \prod_{q\geq 3, q\neq p} \vert A^{(q)} \vert=
  \prod_{q\geq 3, q\neq p} \vert A'^{(q)} \vert=\frac{\vert A' \vert}{\vert A'^{(2)}\vert \vert A'^{(p)}\vert}, $$
  where $q$ runs over all odd primes distinct from $p$.
  Clearly, $\ell $ is an odd integer. For any $N$ as  above, $\zeta_{ {N} \ell} (\chi)=\zeta_{ {N} \ell}
  (\chi')$. Observe that
  $$\zeta_{ {N} \ell} (\chi)=\zeta_{ {N} \ell} (\chi^{(2)})
  \prod_{q\geq 3} \zeta_{ {N} \ell} (\chi^{(q)}),$$
  where $q$ runs over all odd primes. Since $N\ell$ is divisible by
  $2\vert A^{(q)} \vert$ for   $q\neq p$, the   remark at the end of
  Section~\ref{compos} implies that $ \zeta_{ {N} \ell}
  (\chi^{(q)})= \zeta_{ {N} \ell} (\chi'^{(q)}) \neq 0$ for all  $q\neq p$. Replacing if
  necessary $N$ by a bigger power of 2, we can assume that $N$ is
  divisible by $8 \vert A^{(2)} \vert=8 \vert A'^{(2)} \vert$. The last claim of Lemma~\ref{lemma-two} yields $ \zeta_{ {N} \ell}
  (\chi^{(2)})= \zeta_{ {N} \ell} (\chi'^{(2)})=1$. Combining these
  equalities, we obtain that $\zeta_{ {N} \ell} (\chi^{(p)})=\zeta_{ {N} \ell}
  (\chi'^{(p)})$. Expanding both sides as in Formula~\eqref{zetal1}
  and using Formula~\eqref{tok} and
  the inclusions $  \sigma_{ s}
  (\chi^{(p)}) ,   \sigma_{ s} (\chi'^{(p)})\in \{\pm 1\} $, we
  obtain that
  $$\prod_{{\rm {odd}}\,  s\geq 1} \sigma_{ s}
  (\chi^{(p)})= \prod_{{\rm {odd}}\, s\geq 1}  \sigma_{ s} (\chi'^{(p)}).$$
  Replacing in this argument $\ell$ by $\ell p, \ell p^2, \ell p^3 , \ldots $,   we similarly  obtain that for all odd $u\geq 1$ and   even $v\geq 2$,
$$\prod_{{\rm {odd}}\,  s\geq u} \sigma_{ s}
  (\chi^{(p)})= \prod_{{\rm {odd}}\, s\geq u}  \sigma_{ s} (\chi'^{(p)}), \quad
\prod_{{\rm {even}}\,  s\geq v} \sigma_{ s}
  (\chi^{(p)})= \prod_{{\rm {even}}\, s\geq v}  \sigma_{ s} (\chi'^{(p)}) .$$
These equalities easily imply that $\sigma_{ s}
  (\chi^{(p)})= \sigma_{ s} (\chi'^{(p)})$ for all $s$.

\end{document}